\documentclass[a4paper]{article}
\usepackage{fullpage}
\usepackage{enumitem}
\usepackage{amsmath}
\usepackage{amssymb}
\usepackage{bbm}
\usepackage{url}
\usepackage[francais,english]{babel}
\usepackage[latin1]{inputenc}
\usepackage{dsfont}
\usepackage{graphicx}
\usepackage{color}
\usepackage{epsfig}
\usepackage{amsthm,amssymb,amsbsy,amsmath,amsfonts,amssymb,amscd,mathrsfs}
\usepackage{colortbl}
\usepackage{verbatim}
\usepackage{bbm}
\usepackage[multiple]{footmisc}

\newtheorem{remark}{Remark}[section]
\newtheorem{prop}{Proposition}[section]

\newtheorem{corollary}{Corollary}[section]

\newtheorem{lemma}{Lemma}[section]

\newtheorem{theorem}{Theorem}[section]
\newtheorem{hypo}{Hypothesis}[section]

\numberwithin{equation}{section}
\newcommand\be{\begin{equation}}
\newcommand\ee{\end{equation}}
\def\ds{\displaystyle}
\def\Rset{\mathbb{R}}

\title{Optimal control of membrane filtration systems}

\author{
N. Kalboussi\thanks{Université de Carthage, Institut National ds
  Sciences Appliquées et de Technologie \& Université de Tunis El
  Manar, Ecole Nationale d'Ingénieurs de Tunis, Laboratoire de
  Modélisation Mathématique et Numérique dans les sciences d'ingénieur,
  Tunis, Tunisia, {\tt nesrinekalboussi@gmail.com}, {\tt benamar\_nihel@yahoo.fr}, {\tt ellouze\_fatma@yahoo.fr}}
\and
A. Rapaport\thanks{MISTEA, INRA, Montpellier SupAgro, Univ
  Montpellier, Montpellier, France, {\tt alain.rapaport@inra.fr}}
\and
T. Bayen$^{\dagger ,}$\thanks{Institut Montpelliérain Alexander Grothendieck, CNRS, Univ. Montpellier {\tt
    terence.bayen@umontpellier.fr}}
\and
N. Ben Amar$^*$
\and
F. Ellouze$^*$
\and
J. Harmand\thanks{LBE, INRA, Univ Montpellier, Narbonne, France, {\tt jerome.harmand@inra.fr}}
}

\date{September 30, 2017}

\begin{document}

\maketitle

\begin{abstract}
This paper studies an optimal control problem related to membrane
filtration processes. A simple mathematical model of membrane
fouling is used to capture the dynamic behavior of the filtration
process which consists in the attachment of matter onto the membrane
during the filtration period and the detachment of matter during the
cleaning period. We consider the maximization of the net production of a membrane filtration
system (\emph{i.e.} the filtrate) over a finite time horizon, where
control variable is the sequence of filtration/backwashing cycles
over the operation time of process. Based on the Pontryagin Maximum
Principle, we characterize the optimal control strategy and show
that it presents a singular arc. Moreover we prove the existence of
an additional switching curve before reaching the terminal state,
and also the possibility of having a dispersal curve as a locus
where two different strategies  are both optimal.
%The optimal synthesis is obtained by using the Pontryagin Maximum Principle and it
%exhibits a singular arc, a switching curve, and a dispersal curve.\\
\end{abstract}


\section{Introduction}
Membrane filtration systems are widely used as physical separation
techniques in different industrial fields like water desalination,
wastewater treatment, food, medicine and biotechnology. The membrane
provides a selective barrier that separates substances when a
driving force is applied across the membrane.

%The main disadvantages of these processes is the membrane fouling by
%the continuous accumulation of the filtered impurities onto the
%membrane surface (filter cake) and pores.
Different fouling mechanisms are responsible of the flux decline at
constant transmembrane pressure (TMP) or the increase of the TMP at
a constant flux. Hence, the operation of the membrane filtration
process requires to perform regularly cleaning actions like
relaxation, aeration, backwashing and chemical cleaning to limit the
membrane fouling and maintain a good filtrate production.

Usually, sequences of filtration and membrane cleaning are fixed
according to the recommendations of the membrane suppliers or chosen
according to the operator's experience. This leads to high
operational cost and to performances (for example, quantities of
fluid filtered over a given period of time in a membrane
filtration process) that can be far from being optimal. For this
reason, it is important to optimize the membrane filtration process
functioning in order to maximize system performances while
minimizing energy costs.

A variety of control approaches have been proposed to manage
filtration processes. In practice such strategies were based on the
application of a cleaning action (physical or chemical) when either
the flux decline through the membrane or the TMP increase crosses
predefined threshold values \cite{Ferrero2012}. Smith \emph{et
al.} developed a control system that monitors the TMP evolution over
time and initiates a membrane backwash when the TMP exceeds a given
set-point, \cite{Smith1958}.  In \cite{Hong2008} the TMP was also
used as the monitoring variable but the control action was the
increase or decrease of `membrane aeration'. The permeate flux was
used in \cite{Vargas2008} as the control variable to optimize the
membrane backwashing and prevent fouling. Moreover, knowledge-based
controllers found application in the control of membrane filtration
process.  In \cite{Robles2013}, Robles \emph{et al.} proposed an
advanced control system composed of a knowledge-based controller and
two classical controllers (on/off and PID) to manage the aeration
and backwash sequences. The permeability was used by
\cite{Ferrero2011} as a monitoring variable in a knowledge-based
control system to control membrane aeration flow.

To date, different available control systems are able to increase significantly the membrane filtration process performances.
However, more enhanced optimal control strategies are needed to cope with the dynamic operation of the purifying system and to limit membrane fouling.
The majority of the control strategies previously cited address energy
consumption, but regulation and control have not being proved to be optimal.

In the present work, we consider the maximization of the net fluid production (\emph{i.e.} the filtrate) per area of a membrane
filtration system over a given operation duration. 
 The
control variable
is the direction of the flow rate: forward for filtration through
the membrane and backward for backwashing attached foulants.
This problem is quite generic for various fluids to be filtered
%the membrane deposit (in the present work, it is assumed
%that the membrane fouling is only due to the particle deposition onto
%the membrane surface and that the TMP is constant).
Membrane fouling is assumed to be only due to the particle
deposition onto the membrane surface while pores blocking is
neglected. 
This problem appears of primer importance for water treatment, especially in the
actual context of worldwide scarcity of water of `good' quality.

The modeling of the process then leads to consider an optimal
control problem governed by a one-dimensional non-linear dynamics
taking into account the filtration and backwash operating modes.
This optimal control problem falls into the class of classical
Lagrange problems, but in which the Hamiltonian is linear in the
control variable. For such problems, it is well known that several
singularities may occur such as singular arcs, switching surfaces
(\emph{cf.} \cite{bonnard})... The aim of the present work is to
give a complete optimal synthesis of this problem in a quite generic
way ({\it{i.e.}} without giving the exact expressions of the
functions involved in the model) characterizing the occurrence of
such singularities.  The analysis of these singularities is
important for practical implementation because it gives the
structure of the control strategies to be applied (how many
switches, where or when to switch...) and the information
({\it{i.e.}} which variable and when) that is needed to be measured.
%The paper is organized as follows. We first present in Section
%\ref{secprelim} the model and
%gives preliminaries results about the structure of the optimal control
%close to the terminal time. Then Section \ref{secSA} is devoted to the
%analysis of singular arcs (existence and optimality). In Section
%\ref{secSwitch}, we show that there exists other switching locus and moreover
%that a phenomenon of ``dispersion'' may occur. All these singularities
%are illustrated on models from the literature in Section \ref{secEx},
%before giving a conclusion.
%\indent \textcolor{red}{Rajout Térence} The modeling of the process then leads to consider an optimal control problem of Lagrange type
%governed by a one-dimensional non-linear dynamics taking into account the two operations of filtration and backwash.
%The functional here represents the total net amount of water per area of membrane to be maximized over a given time interval.
%The paper is organized as follows:
The paper is organized as follows.
\begin{itemize}
\item In Section \ref{model-sec}, we present the model that allows us to state the optimal control problem
and we give preliminary results about the structure of the optimal control near to the terminal time.
\item Section \ref{singulararc-sec} is devoted to the
analysis of singular arcs (existence and optimality).
%In Section \ref{singulararc-sec}, we show that a {\it{singular arc}} may appear and we give properties of optimal controls.
\item In Section \ref{switchcurve-sec}, we show that a {\it{switching curve}} may appear and moreover
that a phenomenon of `dispersion' may occur. This allows us to
provide a complete
description of an optimal feedback control of the problem (see Theorem \ref{theo-conclusion}).% an optimal synthesis of the problem
%({\it{i.e.}} an optimal feedback control depending both on time and state).
\item Section \ref{appli-sec} depicts the previous analysis on two different models. In the first one,
the optimal synthesis involves both a singular arc and a switching curve only whereas the second one example
also exhibits a {\it{dispersal curve}} (see {\it{e.g.}} \cite{dispersal}).  Such a curve is a locus where the optimal control is non-unique :
the corresponding trajectories (in our case in number $2$) reach the terminal state with the same cost.
\item Finally, several possible extensions of this study are discussed in the conclusion. %in the last section.
\end{itemize}

\section{Model description and preliminary results}\label{model-sec}

To describe the membrane filtration process, we consider a simple
form of the model of \cite{Benyahia2013}. In a previous work, it was
shown that this model is very generic in the sense that it is able
to capture the dynamics of a large number of models available in the
literature while simple enough to be used for optimizing and control
purposes, see \cite{Kalboussi2016}. In the present work, it is
assumed that the membrane fouling is only due to the particle
deposition onto the membrane surface. Let $m$ be the mass of the
cake layer formed during the filtration ($m\geq 0$). One can
assume that $m$ follows a differential equation
$$
\dot m = f_{1}(m),
$$
where $f_1:\Rset_+\rightarrow \Rset_+$.
We further assume that the physical cleaning of the membrane is
performed by a backwashing which consists in reversing the flow.
During this phase, the filtration is stopped and the mass detaches
from the membrane surface with a dynamics
$$
\dot m = -f_{2}(m),
$$
where $f_2:\Rset_+\rightarrow \Rset_+$. The considered system is operated by alternating two functioning
modes: filtration and backwash.
For this reason, we consider a control $u$ that takes values $1$
during filtration and $-1$ during retro washing.
Then, the controlled dynamics can be
written as follows
\begin{equation}
\label{dm}
\dot m= \frac{1+u}{2}f_{1}(m)-\frac{1-u}{2}f_{2}(m) \quad \mathrm{with} \quad m(0)=m_{0},
\end{equation}
where $m_0\geq 0$ is the initial mass. As already mentioned in the introduction, the aim of this work is to
determine an optimal switching between the two functioning modes which
maximizes the net fluid production of the membrane filtration process
during a time interval $[0,T]$.
Assuming that the flux that passes through the membrane during forwards and backwards operation is given by
a function $g:\Rset_+\rightarrow \Rset$ that depends on $m$, the net
amount of fluid per
area of membrane during a time interval
$[0,T]$ is then
\[
J_{T}(m_{0},u(\cdot))=\int_{0}^T u(t)g(m(t)) dt .
\]
Given an initial condition $m_{0}\geq 0$, the objective of the paper is to determine an optimal
strategy $u(\cdot)$ that takes values $-1$ or $1$ maximizing
$J_{T}(m_{0},u(\cdot))$.
%$$
%\max_{u(\cdot)\in \{\pm 1\}} J_{T}(m_{0},u(\cdot))
%$$
Nevertheless, it is well known from the theory of optimal control that
the existence of an optimal trajectory cannot be guaranteed when the
control set is non convex \cite{LeeMarkus1967}. Therefore, we shall consider for the
mathematical analysis that the control $u(\cdot)$ can take values in
the interval $[-1,1]$. Hence, we will focus in this paper on the following optimal control problem:
$$
\max_{u(\cdot)\in \mathcal{U}} J_{T}(m_{0},u(\cdot)),
$$
where $\mathcal{U}$ denotes the set of measurable functions over $[0,T]$ taking values in $[-1,1]$.
The question of practical applicability of a
control that takes values different to $-1$ and $1$ relies on
approximations with chattering controls \cite{ZelikinBorisov1994}
and is exposed in \cite{Kalboussi2017,Kalboussi2017b} (see also \cite{BGM1} in the context of fed-batch bioprocesses).

Next, we consider the following hypotheses on the model.
\begin{hypo}
\label{H1}
The functions $f_{1}$, $f_{2}$ and $g$ are $C^1$ functions such
that
\begin{enumerate}
\item[i.] $f_{1}(m)>0$ and $g(m)>0$  for any $m\geq 0$,
\item[ii.] $f_{2}(0)=0$ and $f_{2}(m)>0$ for $m>0$,
\item[iii.] $f_{1}$ and $g$ are decreasing with $\lim_{m\to+\infty}g(m)=0$,
\item[iv.] $f_{2}$ is increasing.
\end{enumerate}
\end{hypo}
Let us comment about these hypotheses:
\begin{itemize}
\item When a membrane operates in filtration, the resistance to flow
is never null and increases according to the mass $m $ of the cake layer formed on the
membrane surface, which subsequently decreases the permeate flux.
Thus, we assume that the rate $f_{1}$ at which the mass of material
adheres to the membrane surface during filtration is a positive
decreasing function.
\item When starting membrane backwash, the cake layer is
decomposed and the membrane's permeability increases again. So, the
speed $f_{2}$ of the cake detachment can be described by a
positive increasing function. When the membrane is clean ($m=0$),
there is nothing to be detached: $f_{2}(0)=0$.
\item At constant TMP, the permeate flux decreases as the extent of
  fouling gradually increases. Therefore, the variation of the
  permeate flux $J$ can be described by a decreasing
  positive function of the mass of the fouling layer.
\end{itemize}

Thanks to Hypothesis \ref{H1}, one can straightforwardly check the following property.

\begin{lemma}
The domain $\{m>0\}$ is positively invariant whatever is the control $u(\cdot)$.
\end{lemma}

For convenience, we define two functions $f_+:\Rset_+\rightarrow \Rset_+$ and $f_-:\Rset_+\rightarrow \Rset$ defined by
\[
f_{+}(m):=\frac{f_{1}(m)+f_{2}(m)}{2}, \quad f_{-}(m):=\frac{f_{1}(m)-f_{2}(m)}{2},
\]
thus the dynamics can be equivalently written
$$
\dot{m}=f_-(m)+uf_+(m), \quad u\in [-1,1].
$$
We shall use the Maximum Principle of Pontryagin (PMP) \cite{Pontryagin1964}
in order to determine necessary conditions on optimal trajectories.
For this purpose, we introduce the Hamiltonian of the
system defined by\footnote{As the terminal state is free, no abnormal trajectories will occur, moreover we will write
the Hamiltonian condition \eqref{PMP} with a maximum.}
\begin{equation}
\label{Hamiltonian}
H(m,\lambda,u)=\lambda f_{-}(m)+u\left(\lambda f_{+}(m)+g(m)\right).
\end{equation}
According to the Pontryagin Maximum Principle, if $u(\cdot)$ is an optimal control and $m(\cdot)$ the associated trajectory, there exists an absolutely continuous function
$\lambda:[0,T]\rightarrow \Rset$ called adjoint vector satisfying the adjoint equation for a.e. $t\in [0,T]$:
\begin{align}\label{dlambda}
\dot \lambda (t)&= -\frac{\partial H}{\partial m}(m(t),\lambda(t),u(t))\\
&=-\lambda(t) f_{-}'(m(t))-u(t)\left(\lambda(t) f_{+}'(m(t))+g'(m(t))\right), %\quad \mathrm{a.e.} \; t\in [0,T],
\end{align}
together with the terminal condition $\lambda(T)=0$. Moreover, the {\it{Hamiltonian condition}} is satisfied:
\be{\label{PMP}}
u(t)\in \mathrm{arg} \max_{\omega\in [-1,1]} H(x(t),\lambda(t),\omega), \quad \mathrm{a.e.} \; t\in [0,T].
\ee
%Equation \eqref{PMP} is called }. Hence, the optimal control $u$ is such that
%to be maximized w.r.t. $\omega\in [-1,1]$ (thanks to Pontryagin's Principle) at an optimal control $u$ such that
Thanks to this expression, an optimal control necessarily satisfies:
\be{\label{PMP-BB}}
u = \left|\begin{array}{cl}
+1 & \mbox{when } \phi(m,\lambda)>0,\\
-1 & \mbox{when } \phi(m,\lambda)<0,\\
\in [-1,1] &\mbox{when } \phi(m,\lambda)=0,
\end{array}\right.
\ee
where $\phi$ is the {\it{switching function}} defined by
\[
\phi(m,\lambda):=\lambda f_{+}(m)+g(m).
\]
%The adjoint equation given by the PMP is then
%\begin{equation}
%\label{dlambda}
%\dot \lambda = -\partial_{m}H(m,\lambda,u)=
%-\lambda f_{-}'(m)-u\left(\lambda f_{+}'(m)+g'(m)\right),
%\end{equation}
The adjoint vector $\lambda$ satisfies the following property.
\begin{prop}
\label{prop1}
Under Hypothesis \ref{H1}, the adjoint variable satisfies
$\lambda(t)<0$ for any $t \in [0,T[$.
Moreover, for any initial condition $m_{0}$ there exists $\bar t <T$
such that the control $u(t)=1$ is optimal for $t \in [\bar t,T]$.
\end{prop}

\begin{proof}
At $\lambda=0$, one has $\phi(m,0)=g(m)>0$ and then $u=1$ which
implies to have $\dot \lambda=-g'(m)>0$. If $\lambda(t)=0$ for
some $\bar t<T$ then one has necessarily $\lambda(t)>0$ for any
$t>\bar t$ which is in contradiction with $\lambda(T)=0$. Therefore
$t\mapsto \lambda(t)$
is non-null and has constant sign on $[0,T[$. As $\lambda$ has to
reach $0$ at time $T$ with $\dot \lambda(T)>0$, we conclude that $\lambda$
has to be negative on $[0,T[$.
At the terminal time, one has
$\phi(m(T),\lambda(T))=\phi(m(T),0)=g(m(T))>0$. By continuity, the
function $t \mapsto \phi(m(t),\lambda(t))$ remains positive on a time
interval $[\bar t,T]$ with $\bar t<T$, thus we necessarily have $u=1$ on
this interval.
\end{proof}
A triple $(x(\cdot),\lambda(\cdot),u(\cdot))$ is called an extremal trajectory if it satisfies \eqref{dm}-\eqref{dlambda}-\eqref{PMP}.
Since the system and the cost are autonomous ({\it{i.e.}} they do not explicitly depend on the time $t$), the Hamiltonian $H$
is constant along any extremal trajectory.

We call {\it{switching time}} (or switch) an instant $t_s\in [0,T]$
where the optimal control is non-constant in any neighborhood of $t_s$. It follows that at such an instant $t_s$, $\phi$ is necessarily vanishing {\it{i.e.}}
$\phi(t_s)=0$.  We then say that the trajectory has a {\it{switching point}} at the time $t_s$.
As the Hamiltonian $H$ is linear with respect to the control variable, we
know that the optimal solution is a combination of bang-bang controls
and possible singular arcs. Recall that a {\it{singular arc}} is a time interval on which the switching function $\phi$ is identically equal to zero
(see \cite{bonnard,Boscain2005} for a thorough study of this notion).
Since the Hamiltonian is linear w.r.t. the control $u$,
the Hamiltonian condition \eqref{PMP} does not imply straightforwardly an expression of the optimal control as
in \eqref{PMP-BB}. In the two coming sections, we study first
the possibility of having a singular arc, and then the
possibility of having switching points
outside the singular arc.

\section{Singular arc and first optimality results}\label{singulararc-sec}
In this section, we show that singular arcs may appear in the optimal synthesis of the problem.
For convenience, we define a function $\psi:\Rset_+\rightarrow \Rset$ by:% the function
\[
\psi(m):=g(m)\left[f_{-}'(m)f_{+}(m)-f_{-}(m)f_{+}'(m)\right]+g'(m)f_{+}(m)f_{-}(m), \quad m \geq 0.
\]
It will be also convenient to introduce the function $\gamma:\Rset_+\rightarrow \Rset$ defined as
\[
\gamma(m):=-\frac{g(m)f_{-}(m)}{f_{+}(m)} \quad m\geq 0. % \quad m\geq \bar m .
\]
We now consider the following hypothesis:
\begin{hypo}
\label{H2}
The function $\psi$ admits an unique positive root $\bar m$ and is
such that $\psi(m)(m-\bar m)>0$ for any positive $m \neq \bar m$.
\end{hypo}
Under Hypothesis \ref{H2}, one can characterize $m=\bar m$ as the unique
candidate singular arc.
\begin{lemma}
Consider a singular arc defined over a time interval $[t_1,t_2]$. Then the corresponding extremal singular trajectory
$(m(\cdot),\lambda(\cdot),u(\cdot))$ satisfies $m(t)=\bar m$ and $u(t)=\bar u$, $t\in [t_1,t_2]$, where
\begin{equation}
\label{ubar}
\bar u := -\frac{f_{-}(\bar m)}{f_{+}(\bar m)}.
\end{equation}
Moreover, $\lambda(\cdot)$ is constant equal to $\bar \lambda$ where $\bar \lambda\in \Rset$ is defined by
\be{\label{ad-sing}}
\bar\lambda=-\frac{g(\bar m)}{f_{+}(\bar m)}.
\ee
\end{lemma}
\begin{proof}
For simplicity, we write $\dot{\phi}$ the time derivative of $t\mapsto \phi(m(t),\lambda(t))$ and we drop the $m$
dependency of functions $f_{-}$, $f_{+}$ and $g$. Thus, we have:
\[
\begin{array}{lll}
\dot \phi & = & -(\lambda f_{-}' + u(\lambda f_{+}'+g')) +
(\lambda f_{+}'+g')(f_{-}+f_{+}u)\\[2mm]
& = & \lambda(f_{+}'f_{-}-f_{-}'f_{+})+g'f_{-}\\[2mm]
& = &  \ds g(f_{-}'-f_{+}'f_{-}/f_{+})+g'f_{-}+ \phi\,\frac{f_{+}'f_{-}-f_{-}'f_{+}}{f_{+}},
\end{array}
\]
or equivalently
\begin{equation}
\label{dotphi}
\dot \phi =  \frac{\psi}{f_{+}} + \phi\,\frac{f_{+}'f_{-}-f_{-}'f_{+}}{f_{+}}.
\end{equation}
As a singular arc has to fulfill $\phi=0$ and $\dot\phi=0$, thus equation
(\ref{dotphi}) and Hypothesis \ref{H2} imply $\psi=0$. Then, the single possibility for having
a singular arc on a time interval $[t_{1},t_2]$ is to have $m(t)=\bar
m$ for any $t \in [t_{1},t_{2}]$.
From equation (\ref{dm}), one then obtains the constant control given
in (\ref{ubar}) for having $\dot m=0$ at $m=\bar m$. Finally, \eqref{ad-sing} is obtained using that
$\phi=\lambda f_+ + g$ is zero along a singular arc.
\end{proof}
We deduce the following optimality results.
\begin{prop}
\label{prop2}
Suppose that Hypotheses \ref{H1} and \ref{H2} hold true and let $m_0>0$ be an initial condition. Then, the following properties are satisfied:%:%, one has the following properties:
\begin{enumerate}
\item[\emph{(i)}.] When $m_{0}<\bar m$, the control $u=1$ is optimal as long as the corresponding trajectory satisfies $m(t)<\bar m$,
\item[\emph{(ii)}.]  When $m_{0}>\bar m$, either the control $u=1$ is optimal until $t=T$,
or the control $u=-1$ is optimal until a time $\bar t <T$ with $m(\bar
t)\geq \bar m$. If $m(\bar t)>\bar m$ then $u=1$ is optimal on $[\bar t,T]$.
\item[\emph{(iii)}.] Suppose that $f_{-}(\bar m)\geq 0$. Then, for any initial condition $m_0\geq \bar m$, an optimal control satisfies
$u=-1$ over some time interval $[0,\bar t]$ with $\bar t\in [0,T]$ and $u=+1$ over $[\bar t,T]$.
%$u=1$ is optimal on $[\bar t,T]$.. A CHANGER
\item[\emph{(iv)}.] Suppose that $f_{-}(\bar m)<0$ and let $\bar T\in \Rset$ be defined by
\begin{equation}
\label{defbarT}
\bar T:=T-\int_{\bar m}^{\bar m_{T}} \frac{dm}{f_{1}(m)} \; \mbox{ with
} \;
\bar m_{T}:=g^{-1}(\gamma(\bar m)).
%\left(-g(\bar m)\frac{f_{-}(\bar
 % m)}{f_{+}(\bar m)}\right).
\end{equation}
Then, if $\bar T>0$, any singular trajectory is optimal until $t=\bar T$.
%one has Any singular trajectory satisfying $m(t)=\bar m$ for $t\in$
%If $m(t)=\bar m$ with $t<\bar T$ then the singular arc $m=\bar m$ is optimal until $\bar T$, with the constant control

%%If $m(t)\geq \bar m$ with $t\geq \bar T$, then $u=1$ is optimal until $T$.
\end{enumerate}
\end{prop}
\begin{proof}
%From Hypothesis \ref{H2} and equation (\ref{dotphi}), it is worth to mention the
%following properties:% are then satisfied
Let us start by stating two properties that will be crucial for reducing the number of possible switching times in the optimal synthesis.
From Hypothesis \ref{H2} and equation (\ref{dotphi}), we can deduce that:
\begin{itemize}
\item When $\phi(m)=0 \mbox{ with } m<\bar m$ then $\dot\phi<0$. This
  implies that $\phi$ can change its sign only when
  decreasing. Therefore only a switching point from $u=1$ to $u=-1$ can be optimal in
  the domain $\{ m<\bar m\}$.
\item When $\phi(m)=0 \mbox{ with } m>\bar m$ then $ \dot\phi>0$. This implies that $\phi$
can change its sign only when increasing. Therefore, only a switching point from $u=-1$
to $u=1$ can be optimal in  the domain $\{ m>\bar m\}.$
\end{itemize}

Let us now prove (i). Take $m_0<\bar m$, and suppose that the control satisfies $u=-1$. It follows that the trajectory remains
in the domain $\{ m<\bar m\}$.
From Proposition \ref{prop1}, the trajectory necessarily has a switching point at time $t_c$ (otherwise, we would have $u=-1$ until the terminal time
$t=T$ and a contradiction) implying $\dot{\phi}(t_c)\geq 0$. On the other hand, we deduce from \eqref{dotphi}
that $\dot{\phi}(t_c)=\frac{\psi(m(t_c))}{f_+(m(t_c))}<0$ which is a contradiction. Hence, we must have $u=1$ in the domain  $\{ m<\bar m\}$.

The proof of (ii) is similar utilizing that in the domain $\{ m>\bar m\}$, any optimal trajectory has at most one switching point from $u=-1$ to $u=+1$.
It follows that only three cases may occur: either $u=1$ is optimal over $[0,T]$, or the trajectory reaches $m=\bar m$ at some instant $\bar t<T$, or
finally it has exactly one switching point in the domain $\{ m>\bar m\}$ from $u=-1$ to $u=+1$.
%contradicts that $\dot{\phi}(t_c)\geq 0$ since
%then the control $u=-1$ cannot be optimal. Otherwise, we would have $\dot m<0$ and thus the corresponding trajectory
%$m(\cdot)$ enters the domain $\{ m<\bar m\}$ with $u=-1$. Then one has $m(t)<\bar m$
%and $u(t)=-1$ for any future time (as a switching point from $u=-1$ to $u=1$ cannot be
%optimal on this domain), which contradicts Proposition \ref{prop1} and proves (i).
%\medskip

Let us prove (iii). If one has $u=+1$ at time zero, then the result is proved with $\bar t=0$. Suppose now that one has $u=-1$ at time zero.
We know  that if the trajectory switches at some time $\bar t\in [0,T]$ before reaching $m=\bar m$, then one has $u=1$ for $t>\bar t$ and the result is proved.
Suppose now that an optimal trajectory reaches the singular arc before $t=T$ and that one has $m(t)=\bar m$
on a time interval of non-null length.
%Then the adjoint variable has to be constant
%$\lambda=\bar \lambda$ on this time interval, to guarantee $\phi=0$,
%which amounts to have:
Since the Hamiltonian is constant along any extremal trajectory, one must
have $H=\bar\lambda f_{-}(\bar m)$. Moreover, as the Hamiltonian at time $T$
is given by $H=g(m(T))$, one should have $\bar\lambda f_{-}(\bar
m)=g(m(T))> 0$. As $\bar\lambda<0$, we conclude that when $f_{-}(\bar
m)\geq 0$, this situation cannot occur. Hence, a singular arc is not optimal.

Finally, let us prove (iv) and suppose that $f_{-}(\bar m)<0$.
Accordingly to Propositions \ref{prop1} and \ref{prop2}, any optimal
trajectory is such that the corresponding optimal control satisfies $u=1$ in a left neighborhood of $t=T$.
Let us compute the last instant $\bar T<T$ (if it exists) until a singular arc is possible.
From the previous analysis, we necessarily have $u=1$ on $[\bar T,T]$.
%at a certain
%switching time $\bar T<T$ until reaching the terminal time $T$.
This imposes (utilizing that the Hamiltonian is constant) the final state to be $\bar m_{T}=m(T)$ as a solution of
\begin{equation}
\label{barmT}
g(\bar m_{T})=\bar\lambda f_{-}(\bar m)=-\frac{g(\bar m)f_{-}(\bar
  m)}{f_{+}(\bar m)} =\gamma(\bar m),
\end{equation}
which is uniquely defined as $g$ is decreasing, $\lim_{m \rightarrow +\infty} g(m)=0$, and
$-\frac{f_{-}(\bar m)}{f_{+}(\bar m)}\in [0,1]$.
This also imposes that the switching time $\bar T$ can be determined integrating
backward the Cauchy problem
\[
\dot m=f_{1}(m), \quad m(T)=\bar m_{T},
\]
until $m(\bar T)=\bar m$, which amounts to write
\[
\bar T=T-\int_{\bar m}^{\bar m_{T}} \frac{dm}{f_{1}(m)}.
\]
which is exactly the exression \eqref{defbarT}.

We now show that any singular extremal trajectory leaving the singular arc $m=\bar m$ at a time
$t<\bar T$ is not optimal.
To do so, consider a trajectory $m(\cdot)$ leaving the singular arc at a time $t<\bar T$ (necessarily with $u=1$ until the terminal time $T$).
%If not, accordingly to the properties proved
%above, the only possibility to leave the singular arc is to use the
%control $u=1$ until $T$,
In particular, we have $m(T)>\bar m_{T}$.
Since the dynamics is $\dot m=f_{1}(m)$ with $u=1$, the
corresponding cost from time $t$ can be written as follows:
\[
J_{1}(t):=\int_{\bar m}^{m(T)} \frac{g(m)}{f_{1}(m)}dm=\int_{\bar m}^{\bar m_T} \frac{g(m)}{f_{1}(m)}+\int_{\bar m_T}^{m(T)} \frac{g(m)}{f_{1}(m)},
\]
to be compared with the cost $J_s(t)$ of the singular arc strategy from time $t$ ({\it{i.e.}} $u=\bar u$ over $[t,\bar T]$
and then $u=1$ over $[\bar T,T]$), which is equal to
\[
J_{s}(t):=-\frac{g(\bar m)f_{-}(\bar m)}{f_{+}(\bar m)}(\bar T-t)+\int_{\bar
  m}^{\bar m_{T}} \frac{g(m)}{f_{1}(m)}dm .
\]
Thanks to \eqref{defbarT} and using that $T-t=\int_{\bar m}^{m(T)}\frac{dm}{f_1(m)}$, we get
\[
\bar T-t=(T-t)-\int_{\bar m}^{\bar m_{T}} \frac{dm}{f_{1}(m)}=
\int_{\bar m}^{m(T)}\frac{dm}{f_{1}(m)}-\int_{\bar m}^{\bar m_{T}}
\frac{dm}{f_{1}(m)}= \int_{\bar m_{T}}^{m(T)}
\frac{dm}{f_{1}(m)} .
\]
The difference of costs $\delta(m(T))$ can be then written as:
\[
\delta(m(T)):=J_{1}(t)-J_{s}(t)=\int_{\bar m_{T}}^{m(T)} \left(g(m)+\frac{g(\bar
    m)f_{-}(\bar m)}{f_{+}(\bar m)}\right)\frac{dm}{f_{1}(m)}%+ \int_{\bar m_{T}}^{m(T)}\frac{dm}{f_{1}(m)} .
\]
Let us now study the behavior of $\delta$ as a function of $m(T)$.
For convenience, we write $m$ in place of $m(T)$ and recall that $m\geq \bar m_T$ since $m(T) \geq \bar m_T$.
By a direct computation, one has:%compute the two first derivatives of the function $\delta$:
\[
\begin{array}{lll}
\delta'(m) & = & \ds \frac{g(m)+\bar \alpha }{f_{1}(m)},\\[4mm] %\frac{g(\bar m)f_{-}(\bar m)}{f_{+}(\bar m)
\delta''(m)& = & \ds \frac{g'(m)f_1(m)-(g(m)+\bar \alpha)f'_1(m)}{f_{1}^2(m)},%\delta'(m)f_{1}'(m)}.
\end{array}
\]
where $\bar \alpha:=\frac{g(\bar m)f_{-}(\bar m)}{f_{+}(\bar m)}$. From this last expression, since $g'<0$, one has at each $m>0$:
\[
\delta'(m)=0 \; \Longrightarrow \, \delta''(m)<0.
\]
Now, it is to be observed that $\delta(\bar m_{T})=0$ and that $\delta'(\bar m_{T})=0$ (from \eqref{barmT})). The previous analysis then shows that
$\delta'<0$ on $(\bar m_T,+\infty)$. It follows that $\delta$ is decreasing over $[\bar m_T,+\infty)$. Hence, we obtain that
$\delta(m)<0$ for any $m>\bar m_T$.
% we deduce that
%$\delta'(m)<0$ for any $m>m_{T}$. Finally, as $\delta(m_{T})=0$ we
%obtain $\delta(m(T))<0$.
As a conclusion, we have proved that $J_1(t)<J_s(t)$ for any time $t\in [0,\bar T)$,
thus any singular trajectory is such that it is optimal to
stay on the singular locus until $\bar T$ (and then use $u=1$ from
$\bar T$ to $T$) as was to be proved.
\end{proof}
\begin{remark}
In the proof of Proposition \ref{prop2} $\mathrm{(iv)}$, we have pointed out that any singular extremal trajectory is necessarily
optimal until the last possible instant $t=\bar T$.
It is worth to mention that this point is not a consequence of Pontryagin's Principle.
Although no saturation phenomenon of the singular control appears (indeed $u=\bar u$ is constant along the singular locus), singular trajectories
must leave the singular locus at the time $t=\bar T<T$.
%Extremal trajectories could leave the singular arc at some time $t<\bar T$).
\end{remark}

In the sequel, the notation $u[t,m]$
stands for a (non-autonomous) feedback control depending on both
current time $t$ and current state $m$, whereas $u(\cdot)$ denotes
a control function in open loop ({\it{i.e.}} a function of time only chosen for a
given initial condition).
Let us consider the following two sub-domains (that are not disjoined)
\[
{\cal D}_{-} := \{ (t,m) \in [0,T]\times[0,\bar m]\} , \quad
{\cal D}_{+} := \{ (t,m) \in [0,T]\times[\bar m,+\infty)\}.
\]
From Proposition \ref{prop2}, we obtain the following properties about the
optimal control on these two sub-domains. 
\begin{corollary}
\label{cor3}
Under Hypotheses \ref{H1} and \ref{H2}, one has the following properties:
\begin{enumerate}
\item[\emph{(i)}.] If $f_{-}(\bar m)\geq 0$ (where $\bar T$ is defined
in (\ref{defbarT}) when $f_{-}(\bar m)<0$), then the control $u[t,m]=1$ is optimal at any $(t,x)\in {\cal D}_{-}$.
\item[\emph{(ii)}.] If $f_{-}(\bar m)< 0$ and $\bar T\leq 0$ where $\bar T$ is defined
in (\ref{defbarT}), then the control $u[t,m]=1$ is optimal at any $(t,x)\in {\cal D}_{-}$.
\item[\emph{(iii)}.] If $f_{-}(\bar m)< 0$ and $\bar T \in (0,T)$, then the
%then $u=1$ is optimal at any $(t,x)\in {\cal D}_{-}$.
control
\[
u[t,m]=\left|\begin{array}{ll}
1 & \mbox{if } m<\bar m \mbox{ or } t\geq \bar T,\\
\bar u & \mbox{if } m=\bar m \mbox{ and } t<\bar T,
\end{array}\right.
\]
is optimal at any $(t,x)\in {\cal D}_{-}$.
\item[\emph{(iv)}.] The set ${\cal D}_{+}$ is optimally invariant {\it{i.e}}
from any initial condition $(t,m) \in {\cal D}_{+}$, an optimal trajectory stays in
${\cal D}_{+}$ for any future time.
\end{enumerate}
\end{corollary}
\begin{proof}
We have seen that one cannot have $u=-1$ in the domain $\mathcal{D}_-$, otherwise an optimal control cannot be equal to one at the terminal time.
Moreover in the two cases $f_{-}(\bar m)\geq 0$ or $f_{-}(\bar m)< 0$ together with $\bar T\leq 0$, the previous proposition implies that no singular
arc occurs. This proves (i)-(ii). If $f_{-}(\bar m)< 0$ and $\bar T \in (0,T)$, we have seen that singular arcs are optimal
until the terminal time $t=\bar T$. This proves (iii). %e proof of (iii) is straightforward
For proving (iv), we utilize the same argument as for proving (i) and (ii).
\end{proof}
\section{Switching locus and full synthesis}\label{switchcurve-sec}
In this section, we shall provide an optimal synthesis of the problem and we will show in particular that it can exhibit a switching curve depending on the parameter values.
\subsection{Study of the switching locus in $\mathcal{D}_+$}
We start by studying if optimal trajectories can have a switching point.
Accordingly to Proposition \ref{prop2}, this may only occur in the set ${\cal D}_{+}$ with a switching point from
$u=-1$ to $u=1$.  We shall then investigate the locus where switching points occur.
%there is a possibility to have
%a switching point from $u=-1$ to $u=1$ for an optimal trajectory in the domain
%${\cal D}_{+}$ (outside the singular arc).
To do so, in the case where $f_-(\bar m)<0$, consider a parameterized curve $\mathcal{C}$ (possibly empty) contained in $\mathcal{D}_+$ defined by
\be{\label{switching-loc}}
{\cal C}:=\left\{ (\tilde T(\tilde m),\tilde m) \; \vert \; \tilde
m\geq \bar m \mbox{ and } \tilde T(\tilde m)>0 \right\},
\ee
where $\tilde T:[\bar m,+\infty) \rightarrow \Rset$ is the function defined by
\be{\label{switch-time}}
\tilde T (\tilde m):=T-\int_{\tilde m}^{g^{-1}(\gamma(\tilde m))}
\frac{dm}{f_{1}(m)} , \quad \tilde m \geq \bar m.
\ee
The following proposition gives existence and
characterization of this locus contained in ${\cal D}_{+}$.
%For convenience, we define the function
%\[
%\gamma(m)=-\frac{g(m)f_{-}(m)}{f_{+}(m)} , \quad m\geq \bar m .
%\]
\begin{prop}
\label{propDplus}
Assume that Hypotheses \ref{H1} and \ref{H2} are fulfilled.
\begin{enumerate}
\item[\emph{(i)}.] If $f_{-}(\bar m)\geq 0$, then an optimal feedback control is $u[t,m]=1$ for $(t,m)\in {\cal D}_{+}$.
%for any
%initial condition in ${\cal D}_{+}$.
\item[\emph{(ii)}.]  If  $f_{-}(\bar m)< 0$, then:
\begin{itemize}
\item If ${\cal C}$ is empty, an optimal feedback control is $u[t,m]=1$ for $(t,m)\in {\cal D}_{+}$.
\item If ${\cal C}$ is non-empty, consider the domain
\be{\label{set-W}}
{\cal W}:=\left\{ (t,m) \in [0,T)\times]\bar m,+\infty) \; \vert \;
t < \tilde T(m)\right\} .
\ee
Then the feedback control in $[0,T] \times (0,+\infty)$
\be{\label{feed-opti}}
u[t,m]=\left|\begin{array}{rl}
-1 & \mbox{if } (t,m) \in {\cal W}, \\
\bar u & \mbox{if } m=\bar m \mbox{ and } t<\bar T,\\
1 & \mbox{otherwise}.
\end{array}\right.
\ee
is optimal. % for any $(t,m)\in {\cal D}_{+}$.
 Furthermore, the set ${\cal C}$ is tangent to the trajectory that
leaves the singular arc at $(\bar T,\bar m)$ with the control $u=1$.
\end{itemize}
\end{enumerate}
\end{prop}

\begin{proof}

Suppose that $f_{-}(\bar m)\geq 0$ and let us prove (i). We only have to show that any optimal control satisfies
$u=1$ in $\mathcal{D}_+$. In this case, we know that no singular arc occurs,
therefore it is enough to exclude switching points from $u=-1$ to $u=+1$ in $\cal D_+$.
Also, since one has $u=1$ in a neighborhood of $t=T$, it is enough to consider terminal states $m_{T}\geq \bar m$.
%We know from Proposition \ref{prop1} that an optimal trajectory
%reaches the final time with the control $u=1$, thus we denote by $m_{T}$ the
%terminal state of an optimal trajectory.
%
%Suppose first that $f_{-}(\bar m)\geq 0$. We already know that the
%optimal control is $u=1$ at any time if $m_{T}<\bar m$, as the
%trajectory stays in the domain ${\cal D}_{-}$.
%
%When $f_{-}(\bar m)> 0$ and $m_{T}<\bar m_{T}$ (where $\bar m_{T}$ is
%defined in (\ref{defbarT})), the trajectory could not have crossed $m=\bar m$ before $\bar T$
%(by uniqueness of the solution of the dynamics $\dot
%m=f_{1}(m)$ with the control $u=1$) and no switching is possible for
%such optimal trajectories i.e. the constant control $u=1$ is optimal.
%Thu,
%(when $f_{-}(\bar m)\geq
%0$) or $m_{T}\geq \bar m_{T}$ (when $f_{-}(\bar m)<0$),
By integrating
backward the dynamics with the control $u=1$, one has
$H=g(m_{T})=g(m(t))+\lambda(t) f_{1}(m(t))$ for $t<T$ as long as the switching
function
\be{\label{phi-tmp}}
\begin{array}{lll}
\phi(m,\lambda) & = & g(m)+\lambda f_{+}(m)= g(m)+(g(m_{T})-g(m))\frac{f_{+}(m)}{f_{1}(m)}\\
 & = & \ds \frac{f_{+}(m)}{f_{1}(m)}\left(g(m_{T})-\gamma(m)\right),
\end{array}
\ee
remains positive.  As $f_{-}(m)\geq 0$, one has $\gamma(\bar m)\leq
0$.
Notice also that for $m \geq 0$, one has
\begin{equation}{\label{myfun-gamma}}
\gamma'(m)=-\frac{\psi(m)}{f_+(m)^2},%, \quad m \geq 0
\end{equation}
so that $\gamma$ is increasing over $[0,\bar m]$ and decreasing over $[\bar m,+\infty)$.
Since $\gamma$ is decreasing over $[\bar m,+\infty)$
we deduce that $\gamma(m(t))\leq 0$ for any time $t\in [0,T]$. Consequently, $\phi$ cannot change its
sign. Therefore the control $u=1$ is optimal at any time as was to be proved.

Suppose now that $f_{-}(m)<0$ and let us prove (ii). Again, we consider terminal states $m_{T}\geq \bar m$ and we consider the dynamics
with $u=1$ backward in time.
Note that when $m_T=\bar m_T$, then one has $g(\bar m_T)=\gamma(\bar m)$ by conservation of the Hamiltonian.
Consider now an initial state $m_T>\bar m_T$ and the system backward in time with $u=1$.
If an optimal control is such that $u=1$ until reaching the singular arc, we deduce (Thanks to \eqref{phi-tmp}) that
$$
g(m_T)-\gamma(\bar m) < g(\bar m_T)-\gamma(\bar m)=0,
$$
(since $g$ is decreasing). Thus, the switching function is negative when $\bar m$ is reached backward in time by the  trajectory.
%, notice that for $m_{T}=\bar m_{T}$, one has $g(\bar
%m_{T})=\psi(\bar m)$ ($g$ is decreasing), and that for any $m_{T}\geq \bar m_{T}$, $m=\bar m$
%is reached backward in time ($f_{1}(m)$ is strictly positive). Then for $m_{T}>\bar m_{T}$, one has
%$\phi<0$ at $m=\bar m$.
By the mean value Theorem, we conclude that there exists a switching point
that necessarily occurs at some value $\tilde m> \bar m$ such
that $\gamma(\tilde m)=g(m_{T})$, and accordingly to
Proposition \ref{prop2} this switching point (from $u=-1$ to $u=1$) is unique.
%along optimal
%trajectories such that terminal state satisfies $m(T)=m_{T}>\bar
%m_{T}$.
%Compute the derivative of the function $\gamma$:
%\begin{equation}
%\label{gammaprime}
%\gamma'=-\frac{g'f_{-}+gf_{-}'}{f_{+}}+\frac{gf_{-}f_{+}'}{f_{+}^2}
%=-\frac{\psi}{(f_{1}+f_{2})^2} .
%\end{equation}
%Then, by Hypothesis \ref{H2} one has $\gamma'(m)<0$ for $m>\bar m$.
From the monotonicity of $\gamma$ over $[\bar m,+\infty)$, for each $m_T>\bar m$, $\tilde m$ is
uniquely defined by $\tilde m=\gamma^{-1}(g(m_{T}))$, or reciprocally, for any $\tilde m\geq \bar
m$, $m_{T}$ is uniquely defined as a function of
$\tilde m$: $m_{T}(\tilde m)=g^{-1}(\gamma(\tilde m))$ (as $g$ is also a
decreasing invertible function), with
\begin{equation}
\label{mTprime}
m_{T}'(\tilde m)=\frac{\gamma'(\tilde m)}{g'(m_{T}(\tilde m))}\geq 0 .
\end{equation}
Then, the corresponding switching time $\tilde T(\tilde m)$ satisfies
\begin{equation}
\label{Ttilde}
T-\tilde T(\tilde m) = \int_{\tilde m}^{m_{T}(\tilde m)} \frac{dm}{f_{1}(m)} .
\end{equation}
If $\tilde T(\tilde m)\leq 0$ then no switch occurs at $\tilde m$
{\it{i.e.}} the constant control $u=1$ is optimal from $0$ to $T$. It follows that if $\mathcal{C}$ is empty, then $u=1$ is
optimal in $\mathcal{D}_+$ as was to be proved.%, thus also in $[0,T] \times (0,+\infty)$.

When switching points occur, that is, when $\mathcal{C}$ is non-empty, the previous analysis shows that switching points indeed occur
on the curve of $\mathcal{D}_+$ given by \eqref{switching-loc} and the corresponding switching times are given by
\eqref{switch-time} as was to be proved. The optimality of the feedback control \eqref{feed-opti} follows by noting that in $\mathcal{D}_+$,
optimal trajectories have at most one switching point from $u=-1$ to $u=+1$ or from $u=-1$ to $\bar u$.

Finally, the derivative of $\tilde T$ with respect to $\tilde m$ can be
determined from expressions (\ref{Ttilde}) and (\ref{mTprime}) as
\[
\tilde T'(\tilde m)=\frac{1}{f_{1}(\tilde m)}-\frac{m_{T}'(\tilde
  m)}{f_{1}(m_{T}(\tilde m))}=\frac{1}{f_{1}(\tilde
  m)}-\frac{\gamma'(\tilde m)}{g'(m_{T}(\tilde m))f_{1}(m_{T}(\tilde
  m))} .
\]
At $\tilde m=\bar m$, one has $\tilde T(\bar m)=\bar T$ and
$\gamma'(\bar m)=0$ (since $\psi(\bar m)=0)$, which gives $\tilde T'(\bar
m)=1/f_{1}(\bar m) >0$.
Thus, the parameterized curve ${\cal C}$ is %locally a curve about $(\bar T,\bar m)$
indeed tangent to the trajectory that leaves the
singular arc with $u=1$ at $(\bar T,\bar m)$.
\end{proof}

\begin{remark}
When $f_{-}(\bar m)<0$ and $\bar T>0$ (where $\bar T$ is defined in
(\ref{defbarT})), the point $(\bar T,\bar m)$ belongs to the curve ${\cal C}$ which is
then non-empty. This curve could be a set of disjoint curves in $[0,T] \times (0,+\infty)$ (for
instance if the function $\tilde T$ has several changes of sign).
However, in the examples we met, it is always a single curve (bounded or not), see Section \ref{appli-sec}.
Notice also that the map $\tilde m \mapsto \tilde T(\tilde m)$ has no a priori reason to be
monotonic, as one can see in the second example in Section \ref{appli-sec}.
\end{remark}

\subsection{Dispersal curve}
In the sequel, a {\it{switching locus}} is a set of points where optimal trajectories
cross this set by switching from $u=-1$ to $u=+1$. Moreover, the control remains constant equal to one after the switching point and
the corresponding trajectory does not reach the singular arc.
%Then, the and then stay
%equal to $1$ up to the terminal time $T$ without reaching the singular
%arc $m=\bar m$.
On the other hand, a {\it{dispersal curve}} will stand for a set of points from which
there are exactly two optimal trajectories : in our setting, either the optimal control is $u=+1$ until the terminal time or the optimal control is
$u=-1$ in the set $\mathcal{W}$ until the time where the trajectory reaches %reaching
either the singular locus or the switching locus (both strategies having the same optimal cost).

When the set ${\cal C}$ is non-empty (under the condition $f_{-}(\bar m)<0$), we introduce the following partition:
\[
{\cal C} = {\cal C}_{s} \sqcup {\cal C}_{s}
\]
with
\[
{\cal C}_{s}:=\left\{ (t,m)  \in {\cal C} \; ; \;
1+\tilde T'(m)f_{2}(m)>0 \right\} , \quad
{\cal C}_{d}:=\left\{ (t,m) \in {\cal C}\; ; \;
1+\tilde T'(m)f_{2}(m) \leq 0 \right\} .
\]
One can then characterize optimal trajectories on these two sets as follows.
\begin{corollary}
\label{corol}
Assume that Hypotheses \ref{H1}, \ref{H2} are fulfilled with
$f_{-}(\bar m)<0$ and that ${\cal C} \neq \emptyset$ (where ${\cal C}$ is
defined in Proposition \ref{prop2}). One has the following properties.
\begin{itemize}
\item The set ${\cal C}_{s}$ is not
  reduced the singleton $\{(\bar T,\bar m)\}$ and it is a switching locus.
  %Moreover, the optimal control remains equal to $1$ after the switching time and the corresponding trajectory does not reach the singular arc. % reaching the singular arc.
%i.e. optimal trajectories cross ${\cal C}_{s}$ with an optimal
%control switching from $u=-1$ to $u=1$ at ${\cal C}_{s}$
\item The set ${\cal C}_{d}$ (when it is non-empty) is a dispersion locus
{\it{i.e.}} from every state in ${\cal C}_{d}$ the two trajectories
\begin{enumerate}
\item with $u=1$ up to the terminal time,
\item with $u=-1$ up to reaching the singular arc $m=\bar m$ or the
  set ${\cal C}_{s}$,
\end{enumerate}
are both optimal.
\end{itemize}
\end{corollary}

\begin{proof}
The domain ${\cal W}$ (when it is not empty) is exactly the set of points $(t,m)\in \mathcal{D_+}$ for which the optimal
control satisfies $u=-1$ (see Proposition \ref{prop2}).
From such a state, the optimal trajectory has to leave the domain
${\cal W}$ (as $\dot m$ is bounded from above by $-f_{2}(\bar
m)<0$ in this set) reaching either the singular arc or
the set ${\cal C}$. At some point $(t,m)$ in ${\cal C}$, an outward normal $n$ to ${\cal W}$  is then given by
\[
n(t,m)=\left(\begin{array}{c} 1\\ -\tilde T'(m)
\end{array}\right),
\]
and the velocity vectors $v_{-1}$, $v_{1}$ for the control $u=-1$ and $u=1$ respectively
are
\[
v_{-1}(t,m)=\left(\begin{array}{c} 1\\ -f_{2}(m)
\end{array}\right) , \quad
v_{1}(t,m)=\left(\begin{array}{c} 1\\ f_{1}(m)
\end{array}\right) .
\]
Notice that by construction of the set ${\cal C}$, the velocity vector
$v_{1}$ points outward of ${\cal W}$ at any point $(t,m) \in {\cal C}$.
Hence, the velocity vector $v_{-1}$ points outward when the scalar product
$n \cdot v_{-1}$ is positive, that is when $(t,m)$ belongs to ${\cal
  C}_{s}$.

We consider now optimal trajectories that reach ${\cal C}$ from ${\cal W}$
and distinguish two cases.
\begin{enumerate}
\item At states in ${\cal C}_{s}$, the velocity vectors $v_{-1}$,
  $v_{1}$ both point outward of the set ${\cal W}$.
Therefore an optimal trajectory reaching ${\cal C}_{s}$ with $u=-1$
leaves it with $u=1$. Then, accordingly to Proposition \ref{prop1}, the
optimal control stays equal to $1$ up to the terminal time.
\item At states in ${\cal C}_{d}$, $v_{-1}$ points inward of ${\cal
  W}$ while $v_{1}$ points outward. Therefore an optimal
trajectory cannot reach a point located on ${\cal C}_{d}$. From states in  ${\cal
  C}_{d}$, there are thus two extremal trajectories: one
with $u=1$ up to the terminal time, and another one with $u=-1$ up to
the singular arc or to the curve ${\cal C}_{s}$ (accordingly to
Proposition \ref{prop1} and \ref{prop2}) and then $u=1$ up to the terminal time.  As the value function of
a Lagrange problem with smooth data is everywhere Lipschitz continuous (see for
instance \cite{BardiCapuzzo1997}), and that $u=-1$ and $u=+1$ are optimal respectively
inside and outside ${\cal W}$, we deduce that these two extremal trajectories
should have the same cost {\it{i.e.}} are both optimal.
\end{enumerate}
Finally, let us show that ${\cal C}_{s}$ is not
  reduced a singleton. The state $(\bar m,\bar T)$ belongs to ${\cal C}$
  (as it is indeed a point where the switching function vanishes) but
it also belongs to the singular locus $m=\bar m$.
Therefore, there exists a trajectory with $u=-1$ that is
crosses ${\cal C}$ transversely at this point. %$(\bar m,\bar T)$.
By continuity of the
solutions of the system with $u=-1$ w.r.t. the initial
condition, we deduce that there exist locally other trajectories that
cross the non-empty curve ${\cal C}$ transversely with the control $u=-1$.
This proves that ${\cal C}_{s}$ is not reduced to a singleton.
\end{proof}

Figure \ref{figpartitionC} illustrates the two kind of points that
can belong to the set ${\cal C}$.
\begin{figure}[h]
\begin{center}
\includegraphics[width=4cm]{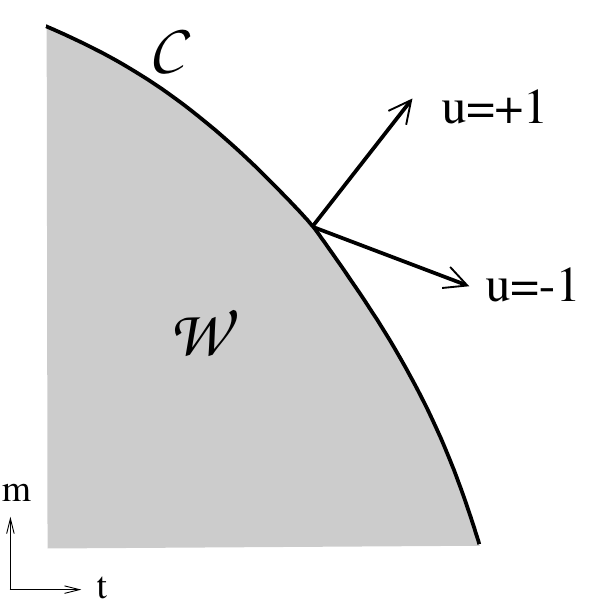}
\hspace{15mm}
\includegraphics[width=4cm]{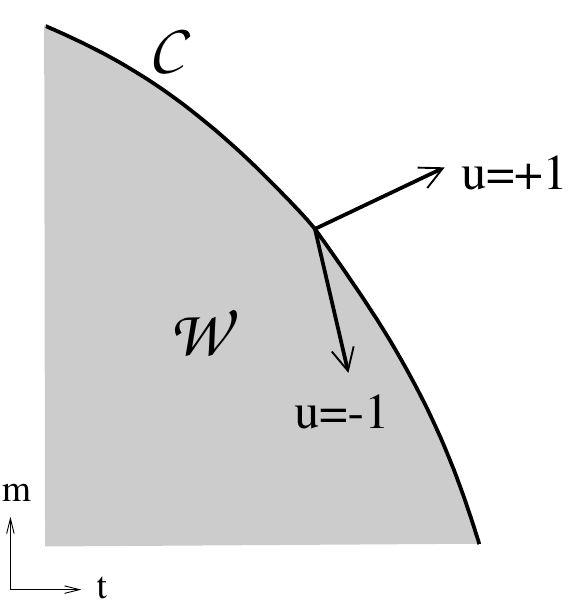}
\caption{Switching point (left) versus dispersion point (right) on the
  set ${\cal C}$.}
\label{figpartitionC}
\end{center}
\end{figure}
\subsection{Full synthesis}
We conclude this section by summarizing the results of Corollary \ref{cor3} and Proposition \ref{propDplus}
that give the optimal synthesis of the problem in the whole domain
$[0,T]\times[0,+\infty)$.
\begin{theorem}{\label{theo-conclusion}}
Assume that Hypotheses \ref{H1} and \ref{H2} are fulfilled.
\begin{enumerate}
\item[\emph{(i)}.] If $f_{-}(\bar m)\geq 0$ or $f_{-}(\bar m)< 0$ and the set $\mathcal{C}$ is empty, then, an optimal feedback control
in $[0,T] \times (0,+\infty)$ is given by $\mathrm{(}$recall \eqref{set-W}$\mathrm{)}$
$$u[t,m]=1.$$
\item[\emph{(ii)}.] If $f_{-}(\bar m)< 0$ and the set $\mathcal{C}$ non-empty, then, an optimal feedback control in $[0,T] \times (0,+\infty)$ is given by
\be{}
u[t,m]=\left|\begin{array}{rl}
-1 & \mbox{if } (t,m) \in {\cal W}, \\
\bar u & \mbox{if } m=\bar m \mbox{ and } t<\bar T,\\
1 & \mbox{otherwise}.
\end{array}\right.
\ee
\end{enumerate}
\end{theorem}
The vacuity of $\mathcal{C}$ can be verified thanks to the explicit definitions of $\mathcal{C}$ by \eqref{switching-loc}.
\section{Two numerical case studies}
\label{appli-sec}
In this section, we illustrate the previous analysis of optimal trajectories
on two classical models of the literature that fulfill
Hypotheses \ref{H1} and \ref{H2}.
\subsection{Benyahia {\em et al} model}
\label{secModel1}
Consider the following functions that have been validated on
experimental data \cite{Benyahia2013}:
\[
f_{1}(m)=\frac{b}{e+m} , \quad f_{2}(m)=am, \quad g(m)=\frac{1}{e+m},
\]
where $a$, $b$ and $e$ are positive numbers.
One can check that Hypothesis \ref{H1} is fulfilled.
A straightforward computation of the function $\psi$ gives
\[
\begin{array}{lll}
\psi(m) & = & \displaystyle
\frac {-1}{2(e+m)}\left[\left( {\frac {b}{ \left( e+m \right) ^{
2}}}+a\right)  \left({\frac {b}{e+m}}+am \right) +
 \left({\frac {b}{e+m}}-am \right)  \left( {\frac {b
}{ \left( e+m \right) ^{2}}}-a \right)\right] \\[4mm]
& & \displaystyle -{\frac {1}{
 2\left( e+m \right) ^{2}} \left({\frac {b}{e+m}}+am
 \right)  \left({\frac {b}{e+m}}-am \right) }\\[6mm]
& = & \displaystyle \frac {{a}^{2}{e}^{2}{m}^{2}+2\,{a}^{2}e{m}^{3}+{a}^{2}{m}^{4}-2
\,ab{e}^{2}-6\,abem-4\,ab{m}^{2}-{b}^{2}}{ 4\left( e+m \right) ^{4}}.

\end{array}
\]
A further computation of the derivative of $\psi$ gives
\[
\psi'(m)=
\frac {{a}^{2}{e}^{3}m+2\,{a}^{2}{e}^{2}{m}^{2}+{a}^{2}e{m}^{3}+
ab{e}^{2}+5\,abem+4\,ab{m}^{2}+2\,{b}^{2}}{ 2\left( e+m \right) ^{5}}.
\]
which allows to conclude that $\psi$ is increasing on $\Rset_{+}$.
As one has $\psi(0)=-(2abe^2+b^2)/(4e^4)<0$ and
$\lim_{m\to+\infty}\psi(m)=+\infty$, we deduce that Hypothesis
\ref{H2} is fulfilled.
When $\psi$ is null for $m=\bar m$, one has
\[
d(\bar m)=f_{-}'(\bar m)f_{+}(\bar m)-f_{-}(\bar m)f_{+}'(\bar m)
=\frac{-g'(\bar m)f_{+}(\bar m)}{g(\bar m)}f_{-}(\bar m).
\]
Therefore $f_{-}(\bar m)$ and $d(\bar m)$ have the same sign.
A straightforward computation gives
\[
d(m)=-\frac {ab \left( e+2\,m \right) }{ 2\left( e+m \right)
    ^{2}}<0,
\]
and thus one has $d(\bar m)<0$.
Therefore, from Proposition \ref{prop2} and
Corollary \ref{corol}, there exists a singular arc when
$\bar T>0$ and a switching locus  when $\tilde T(\tilde m)>0$.\\

Figure \ref{fig:Synthesis} shows the general synthesis of optimal
controls with the parameters $a=b=e=1$ and for a time horizon of $10$
hours.
In this example one can see that the curve ${\cal C}$ is entirely a
switching locus {\it{i.e.}} one has ${\cal C}={\cal C}_{s}$.
\begin{figure}[h]
\begin{center}
\includegraphics[width=13cm]{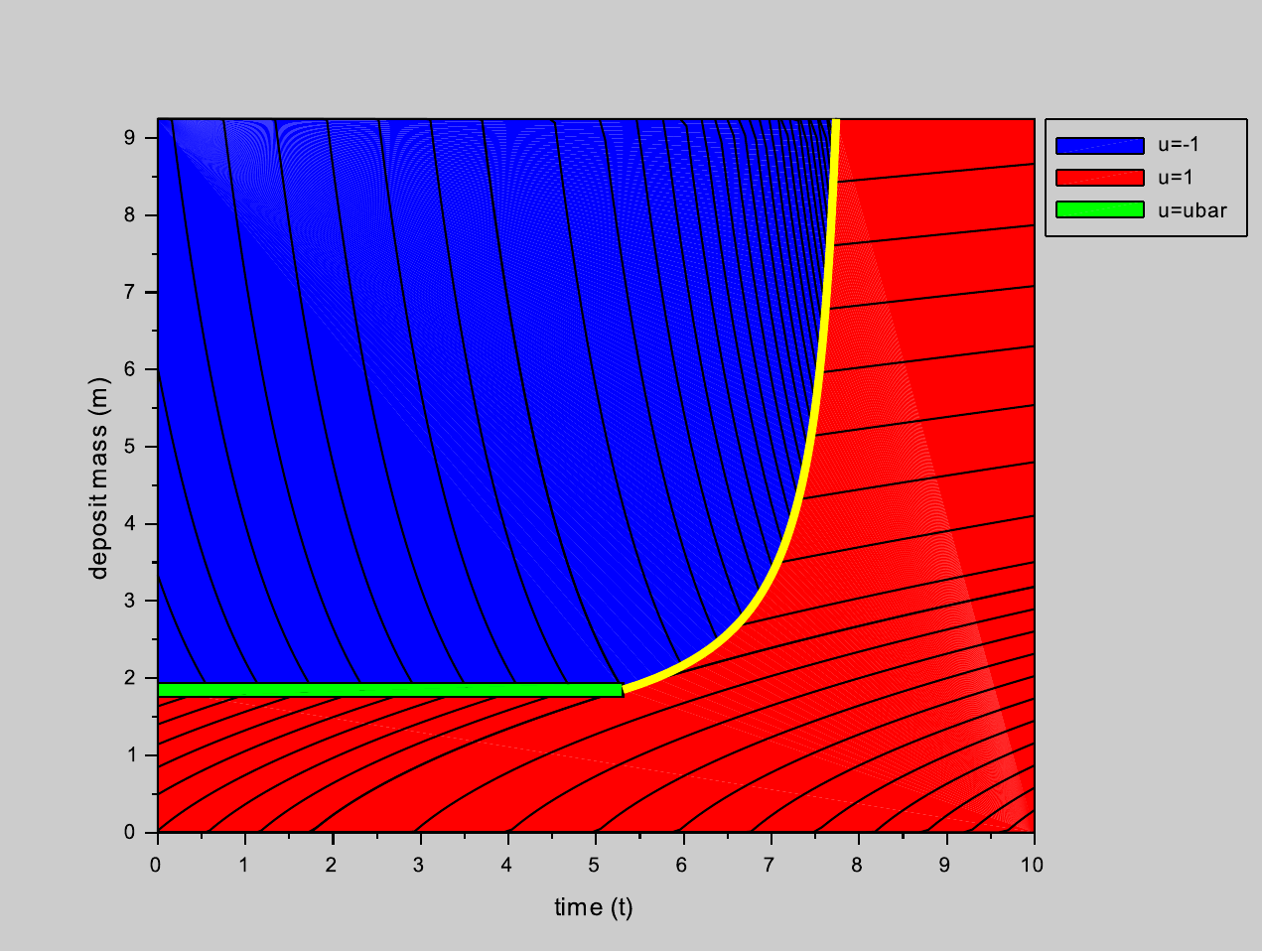}
\caption{Optimal synthesis for the model of Section \ref{secModel1}
  with $a=b=e=1$ and $T=10$ hours. The set ${\cal W}$ is depicted in blue and
  in yellow the switching locus.}
\label{fig:Synthesis}
\end{center}
\end{figure}

\subsection{Cogan-Chellam model}
\label{secModel2}
We now consider the functions
\[
f_{1}(m)=\frac{b}{e+m} , \quad f_{2}(m)=\frac{am}{e+m}, \quad g(m)=\frac{1}{e+m},
\]
where $a$, $b$ and $e$ are positive numbers, as proposed in \cite{Cogan2014,Cogan2016}. Clearly Hypothesis
\ref{H1} is fulfilled. Moreover, one has
\[
\begin{array}{lll}
\psi(m) & = & \displaystyle
-\frac{(ae+b)(b+am)+(ae-b)(b-am)}{4(e+m)^4}
- \frac{(b+am)(b-am)}{4(e+m)^4}\\[4mm]
& = & \displaystyle \frac{a^2m^2-2abe-2abm-b^2}{4(e+m)^4}=
\frac{(am-b)^2-2abe-2b^2}{4(e+m)^4}.
\end{array}
\]
Therefore, the function $\psi$ can have two changes of sign at
\[
\bar m_{-}=\frac{b-\sqrt{2b^2+2abe}}{a} , \quad \bar
m_{+}=\frac{b+\sqrt{2b^2+2abe}}{a},
\]
where $\bar m_{-}$ and $\bar m_{+}$ are respectively negative and positive
numbers. One has also
\[
\psi'(m)=\frac {{a}^{2}em+abe+abm}{
 2\left( e+m \right) ^{5}},
\]
which is positive. Therefore $\psi$ is an increasing function and
Hypothesis \ref{H2} is fulfilled with $\bar m=\bar m_{+}$.
Moreover one can write
\[
f_{-}(\bar m)=\frac{-\sqrt{b^2+2abe}}{e+\bar m}<0 .
\]
Then, as for the previous model, Proposition \ref{prop2} and
Corollary \ref{corol} allow to conclude that there exists a singular arc when
$\bar T>0$ and a switching locus when $\tilde T(\tilde m)>0$.\\

Figure \ref{fig:SynthesisCogan} shows the synthesis of optimal controls for
this model with the parameters $a=b=e=1$ and for a time horizon of $40$ hours.
In this example one can see that the curve ${\cal C}$ splits
into two non-empty subsets ${\cal C}_{s}$ and ${\cal C}_{d}$.
\begin{figure}[h]
\begin{center}
\includegraphics[width=13cm]{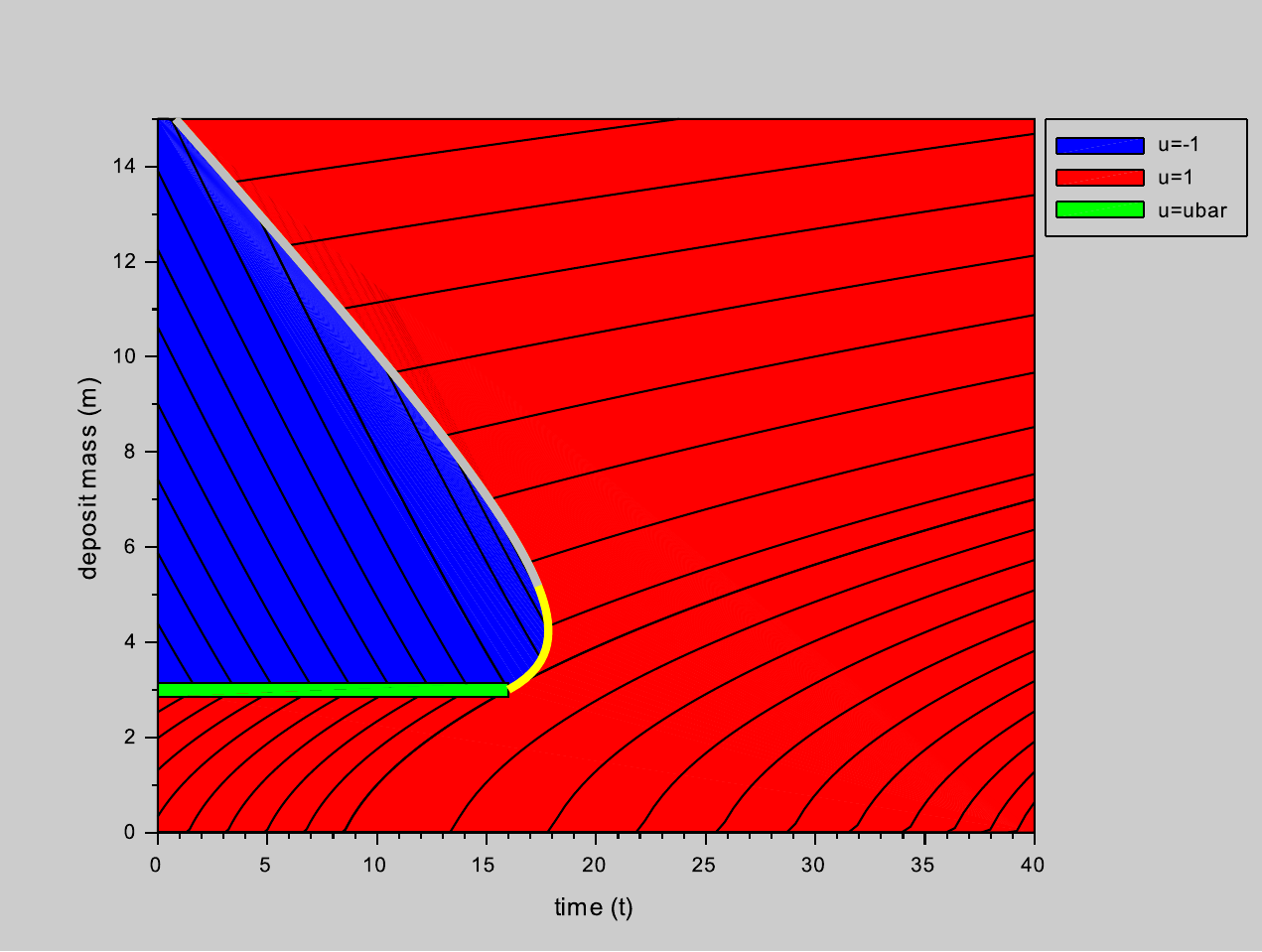}
\caption{Optimal synthesis for the model of Section \ref{secModel2}
  with $a=b=e=1$ and
  $T=40$ hours. The set ${\cal W}$ is depicted in blue,
in yellow the switching locus, and in gray the dispersion locus.}
\label{fig:SynthesisCogan}
\end{center}
\end{figure}

\subsection{Discussion}
Although the two models are very close %mathematical expressions
and posses similar optimal syntheses, a main difference occurs on
the size and on the shape of the domain ${\cal W}$ where backwash
has to be applied (see Figures
\ref{fig:Synthesis},\ref{fig:SynthesisCogan}). In particular, its
boundary ${\cal C}$ is entirely a switching curve in one case while
most of it is a dispersal curve in the second case. This should give
valuable information to the practitioners about when and how long
backwashing ({\it{i.e.}} $u=-1$) has to be applied out of the
singular arc.

For the practical implementation of the optimal control law (where
only the values $u=-1$ and $u=1$ can be applied) it is not possible
to stay exactly on the singular arc $m=\bar m$. But an approximation
by a sequence of filtration/backwashing can be applied to stay on
the vicinity of the singular arc. This sequence can be chosen so
that the average value of $m$ is $\bar m$, which provides a good
approximation of the optimal value as it has been tested in
\cite{Kalboussi2017,Kalboussi2017b}. One may argue that the optimal
control problem could be reformulated as a discrete time problem
where the time step is the smallest period of switching between
filtration and backwashing that could be applied in practice. We
believe that this approach gives less geometric insights of the
nature of the optimal control than the continuous formulation.
Moreover, computing the optimal value of the criterion for the
continuous time problem gives an upper bound of what could be
intrinsically expected from the process, independently of the
practical implementations.

\section{Conclusion}
In this work, the application of the Pontryagin Maximum Principle
for the synthesis of optimal control of a switched system shows
interesting results for maximizing the net fluid production (per
filtrate) of a membrane
filtration system. The optimal synthesis exhibits bang-bang controls
with a `most rapid approach' to a singular arc and a switching curve
before reaching the final time. We have also shown that a dispersal
curve may occur, leading to the non-uniqueness of optimal
trajectories. Practically, the determination of the singular arc
allows to compute a sequence of filtration/backwashing to stay about
the singular arc, and the determination of the curve ${\cal C}$
provides the information about the domain where backwashing has to
be applied. The synthesis also reveals that if one wants to
implement a feedback controller, which is more robust than an
open-loop controller, the on-line measurement of the mass deposit
$m$ or of any invertible function of $m$, such as the fluid
flowrate, is crucial.

The main advantage of the present analysis is to describe an optimal
synthesis for a very large class of models relying on simple
qualitative properties of the functions $f_1$, $f_2$ and $g$.

Perspectives of this work are first to implement
the optimal synthesis with real process constraints, and
compare the fluid production (per filtrate) of the membrane filtration process with
the classical operating strategies that are proposed in the literature
and currently used. Extensions to other fluids or non constant TMP and consideration
of multiple objectives (production and energy consumption) could be also
the matter of future works, as well as possibilities of multiple singular arcs.

\section*{Acknowledgments}
The authors thank the TREASURE research network (cf \url{http://www.inra.fr/treasure}) for its financial support.
The third author would like to thank the MIA Division of INRA and
the MISTEA lab, Montpellier  for providing him a half year delegation during the academic year 2017-2018.

%%%%%%%%%%%%%%%%%%%%%%%%%%%%%%%%%%%%%%%%%%%%%%%%%%%%%%%%%%%%%%%%%%%%%%%%%


\begin{thebibliography}{99}

\bibitem{BardiCapuzzo1997} {\sc M. Bardi, I. Capuzzo-Dolcetta}, {\em Optimal Control and Viscosity Solutions of Hamilton-Jacobi-Bellman Equations}, Modern Birkhaüser Classics, 1997.

\bibitem{BGM1} {\sc T. Bayen, F. Mairet}, {\em Minimal time control of fed-batch bioreactor with product inhibition}
Bioprocess and Biosystems Engineering, vol. 36, 10, pp. 1485-1496, 2013.

%\bibitem{BardiCapuzzo1997}
%Bardi M. and Capuzzo-Dolcetta, I. (1997)
% \newblock \emph{ Optimal Control and Viscosity Solutions of
%   Hamilton-Jacobi-Bellman Equations}
%\newblock Birkh\"auser.

\bibitem{Benyahia2013} {\sc B. Benyahia, A. Charfi, N. Benamar, M. H\'eran, A. Grasmick, B. Cherki, J. Harmand}
 A simple model of anaerobic membrane bioreactor for control design: coupling the ``AM2b'' model with a simple membrane fouling dynamics,
13. World Congress on Anaerobic Digestion: Recovering (bio) Resources for the World. AD13, 2013.

\bibitem{bonnard} {\sc B. Bonnard, M. Chyba}, {\em Singular Trajectories and their role in Control Theory}, Springer, SMAI, vol. 40, 2002.%ok

\bibitem{Boscain2005}{\sc U. Boscain, B. Piccoli}, {\em Optimal Syntheses for Control Systems on 2-D Manifolds}, Springer SMAI, vol. 43, 2004.

%\bibitem{Boscain2005} Boscain, U. and Piccoli, B. (2005) \newblock \emph{Optimal Syntheses for Control
%Systems on 2-D Manifolds}, \newblock Springer SMAI, 43.

\bibitem{Cogan2014}
{\sc N.G. Cogan, S. Chellamb}, {\em A method for determining the optimal back-washing frequency and duration for dead-end microfiltration},
Journal of Membrane Science, 469, pp.410--417, 2014.

\bibitem{Cogan2016}
{\sc N.G. Cogan, J. Li, A.R. Badireddy, S. Chellamb}, {\em Optimal backwashing in dead-end bacterial microfiltration with irreversible attachment},
Journal of Membrane Science, 520, pp.337--344, 2016.

\bibitem{Ferrero2011}
{\sc G. Ferrero, H. Moncl{\'u}s, G. Buttiglieri, J. Comas, I. Rodriguez-Roda},
{\em Automatic control system for energy optimization in membrane bioreactors},
Desalination, 268(1), pp. 276--280, 2011.

\bibitem{Ferrero2012}
{\sc G. Ferrero, I. Rodriguez-Roda, J. Comas}, {\em Automatic control systems for submerged membrane bioreactors: A state-of-the-art review.}
water research, Elsevier, 46, pp. 3421--3433, 2011.

\bibitem{Kalboussi2016}
{\sc N. Kalboussi, J. Harmand, N. Ben Amar, F. Ellouze},
{\em A comparative study of three membrane fouling models - Towards a generic model for optimization purposes},
CARI'2016, 10--16 Oct., Tunis, Tunisia, 2016.

\bibitem{Kalboussi2017}
{\sc N. Kalboussi, A. Rapaport, T. Bayen, N. Ben Amar, F. Ellouze, J. Harmand},
{\em Optimal control of a membrane filtration system}, IFAC 2017, 20th
World Congress of the International Federation of Automatic Control,
2017.

\bibitem{Kalboussi2017b}
{\sc N. Kalboussi,  J. Harmand, A. Rapaport, T. Bayen, N. Ben Amar, F. Ellouze, J. Harmand},
{\em Optimal control of physical backwash strategy - towards the
  enhancement of membrane filtration process performance}, Journal of
Membrane Science, 2017 {\em (in press)}.

\bibitem{Hong2008}
{\sc S.N. Hong, H.W. Zhao, R.W. DiMassimo}, {\em A new approach to backwash initiation in membrane systems},
U.S. Patent No. 7,459,083. Washington, DC: U.S. Patent and Trademark Office, 2008.

\bibitem{LeeMarkus1967} {\sc E.B. Lee, L. Markus}, {\em Foundations of Optimal Control Theory}, SIAM series in applied mathematics, Wiley, 1967.

\bibitem{Pontryagin1964} {\sc L. Pontryagin,  V. Boltyanski, R. Gamkrelidze, E. Michtchenko}, {\em The Mathematical Theory of Optimal Processes},
The Macmillan Company, 1964.

%\bibitem{Pontryagin1964} Pontryagin, L.S., Boltyanskiy, V.G.,
%Gamkrelidze, R.V. and Mishchenko, E.F. (1964) \newblock \emph{Mathematical Theory of
%Optimal Processes}, \newblock \emph{The Macmillan Company.}

\bibitem{Robles2013} {\sc A. Robles, M.V. Ruano, J. Ribes, J. Ferrero},
{\em Advanced control system for optimal filtration in submerged anaerobic MBRs (SAnMBRs)},
Journal of Membrane Science, 430, pp. 330--341, 2013.

\bibitem{Smith1958}
{\sc P.J. Smith, S. Vigneswaran, H.H. Ngo, R. Ben-Aim, H. Nguyen}, {\em A new approach to backwash initiation in membrane systems},
Journal of Membrane Science, 278(1), pp. 381--389, 2006.

\bibitem{Vargas2008}
{\sc A. Vargas, I. Moreno-Andrade, G. and Buitr\'on},
{\em Controlled backwashing in a membrane sequencing batch reactor used for toxic wastewater treatment},
Journal of Membrane Science, 320(1), pp. 185--190, 2008.

\bibitem{ZelikinBorisov1994} {\sc M. Zelikin, V.Borisov}, {\em Theory of Chattering Control with applications to
Astronautics, Robotics, Economics, and Engineering}, Systems \& Control: Foundations \& Applications, Birkhäuser, 1994.

\bibitem{dispersal} {\sc F. L. Chernous'ko, S. A. Reshmin, I. M. Ananievski},
{\em Control of Nonlinear Dynamical Systems: Methods and Applications}, Springer Berlin Heidelberg, 2009.

\end{thebibliography}
\end{document}